\documentclass[a4paper,10pt,twoside]{amsart}
\usepackage{amsthm}
\usepackage{amsmath}
\usepackage{amssymb}
\usepackage{amsfonts}
\usepackage{amscd}
\usepackage[english]{babel}
\usepackage{stmaryrd}
\usepackage{enumerate}

\newtheorem{thm}{Theorem}

\newtheorem{prop}{Proposition}
\newtheorem*{rem}{Remark}

\newtheorem*{defi}{Definition}
\newtheorem{lem}{Lemma}

\begin{document}

    \title{Commuting exponentials in dimension at most 3}
    \author{Gerald BOURGEOIS}
    
    \date{11-7-2011}
    \address{G\'erald Bourgeois, GAATI, Universit\'e de la polyn\'esie fran\c caise, BP 6570, 98702 FAA'A, Tahiti, Polyn\'esie Fran\c caise.}
    \email{gerald.bourgeois@upf.pf}
        
  \subjclass[2010]{Primary 15A24}
    \keywords{Matrix exponential, Property L}

\begin{abstract}
Let $A,B$ be two square complex matrices of dimension at most 3. We show that the following conditions are equivalent\\
i) There exists a finite subset $U\subset\mathbb{N}_{\geq{2}}$ such that
for every $t\in\mathbb{N}\setminus{U}$, $\exp(tA+B)=\exp(tA)\exp(B)=\exp(B)\exp(tA)$.\\
  ii) The pair $(A,B)$ has property L of Motzkin and Taussky and $\exp(A+B)=\exp(A)\exp(B)=\exp(B)\exp(A)$.    
\end{abstract}

\maketitle

    \section{Introduction}
    \textbf{Notation}. 
    $i)$ We denote by $\mathbb{N}$ the set of positive integers and, if $n\in\mathbb{N}$, by $I_n,0_n$ the identity matrix and the zero-matrix of dimension $n$.\\
    $ii)$ If $X$ is a complex square matrix, then $s(X)$ refers to the spectrum of $X$.\\
       \begin{defi}
  The $n\times{n}$ complex matrices $A,B$ are said to be simultaneously triangularizable ($ST$) if there exists an $n\times{n}$ invertible matrix $P$ such that $P^{-1}AP$ and $P^{-1}BP$ are upper triangular matrices.   
    \end{defi}
  In \cite{1}, the author dealed with square matrices $A,B\in\mathcal{M}_n(\mathbb{C})$, $(n=2$ or $3$), satisfying
  \begin{equation} \label{exp} \text{for every } t\in\mathbb{N}, \exp(tA+B)=\exp(tA)\exp(B)=\exp(B)\exp(tA) \end{equation}
  The author concluded that these matrices are simultaneously triangularizable.\\
  The result is true for $n=2$. However it is false for $n=3$. Indeed J.L. Tu communicated to the author the following counter-example
  \begin{equation}  \label{tu} A_0=2i\pi\begin{pmatrix}1&0&0\\0&2&0\\0&0&0\end{pmatrix},B_0=2i\pi\begin{pmatrix}2&1&1\\1&3&-2\\1&1&0\end{pmatrix}.  \end{equation} 
 Clearly $A_0,B_0$ are not $ST$. However it is easy to see that, for every $t\in\mathbb{C}$, the eigenvalues of $tA_0+B_0$ are the entries of its diagonal. Moreover, for every $t\in\mathbb{N}$, the eigenvalues of $tA_0+B_0$ belong to $2i\pi\mathbb{Z}$ and are distinct. Therefore, for every $t\in\mathbb{N}$, 
 $$\exp(A_0)=\exp(B_0)=\exp(tA_0+B_0)=I_3.$$    
  \begin{defi} \cite [Property L]{4} A pair $(A,B)\in\mathcal{M}_n(\mathbb{C})^2$ has property L if there exist orderings of the eigenvalues $(\lambda_i)_{i\leq{n}},(\mu_i)_{i\leq{n}}$ of $A,B$ such that for all $(x,y)\in\mathbb{C}^2$, $s(xA+yB)=(x\lambda_i+y\mu_i)_{i\leq{n}}$. 
  \end{defi}
  \begin{rem}
  If $A,B$ are $ST$, then the pair $(A,B)$ has property L. The converse is false but for $n=2$ (see \cite{4}).
  \end{rem}
  Recently, in \cite[Proposition 5]{3}, C. de Seguins Pazzis proved this result
  \begin{prop} \label{seguins} Let $A,B\in\mathcal{M}_n(\mathbb{C})$ satisfying Condition (\ref{exp}). The pair $(A,B)$ has property L.
  \end{prop}
  We are interested in the converse of Proposition \ref{seguins}. We can wonder whether the conditions $e^Ae^B=e^Be^A=e^{A+B}$ and $(A,B)$ has property L imply Condition (\ref{exp}). The answer is no as the following shows\\
 \textbf{Counter-example.}  The pair $(A_0,-2B_0)$ (cf Example (\ref{tu})) has property L and $\exp(A_0)=\exp(-2B_0)=I_3$. Moreover, for $t\in\mathbb{N}\setminus\{2,3,4\}$, one has $\exp(tA_0-2B_0)=I_3$, and for $t\in\{2,3,4\}$, one has not. Therefore Condition (\ref{exp}) does not hold for this pair.\\
 Thus, we weaken Condition (\ref{exp}) as follows\\ 
  There exists a finite subset $U\subset\mathbb{N}_{n\geq{2}}$ such that
  \begin{equation} \label{expbis} \forall{t}\in\mathbb{N}\setminus{U},\exp(tA+B)=\exp(tA)\exp(B)=\exp(B)\exp(tA). \end{equation}
 \indent In this paper, we show that, in dimension $2$ or $3$, the pair $(A,B)$ satisfies Condition (\ref{expbis}) if and only if $e^{A+B}=e^Ae^B=e^Be^A$ and $(A,B)$ has property L. 
  \section{Property L and Condition (\ref{expbis})} 
    
 The following is a partial converse of Proposition \ref{seguins}. 
  \begin{prop}
  Assume that $A=diag(\lambda_1,\cdots,\lambda_n)\in\mathcal{M}_n(\mathbb{C})$ has $n$ distinct eigenvalues in $2i\pi\mathbb{Z}$, that $B=[b_{ij}]\in\mathcal{M}_n(\mathbb{C})$ (where for every $i\leq{n}$, $b_{ii}\in{2}i\pi\mathbb{Z}$) is diagonalizable and that the pair $(A,B)$ has property L. Then the pair $(A,B)$ satisfies Condition (\ref{expbis}).
  \end{prop}
  \begin{proof}  
   Note that $e^A=I_n$. According to \cite[Theorem 1]{4}, for every $t\in\mathbb{C}$, $$s(tA+B)=(t\lambda_i+b_{ii})_{i\leq{n}}.$$
    Thus $e^{B}=I_n$. For almost all $t\in\mathbb{N}$, $tA+B$ has $n$ distinct eigenvalues in $2i\pi\mathbb{Z}$ and $\exp(tA+B)=I_n$.  
  \end{proof}
  \begin{defi}
  $i)$ Let $A\in\mathcal{M}_n(\mathbb{C})$. The spectrum of $A$ is said to be $2i\pi$ congruence-free ($2i\pi$ CF) if, for all $\lambda,\mu\in{s}(A)$, $\lambda-\mu\notin{2}i\pi\mathbb{Z}^*$.\\
   $ii)$ For every $z\in\mathbb{C}$, $\Im(z)$ denotes its imaginary part.\\
  $iii)$ Let $\log:GL_n(\mathbb{C})\rightarrow\mathcal{M}_n(\mathbb{C})$ be the (non continuous) primary matrix function (cf. \cite{9}) associated to the principal branch of the logarithm, defined by $\Im(log(z))\in(-\pi,\pi]$, for every $z\in\mathbb{C}^*$. Thus for every $X\in{G}L_n(\mathbb{C})$, $s(\log(X))\subset\{z\in\mathbb{C}|\Im(z)\in(-\pi,\pi]\}$.
  \end{defi}
  \begin{lem}    \label{MN}   
Let $A\in\mathcal{M}_n(\mathbb{C})$. There exists a unique pair $(F,\Delta)$ of square $n\times{n}$ matrices, that are polynomials in $A$, such that 
 $$A=F+\Delta,e^F=e^A,e^{\Delta}=I_n\;\text{ and for all }\lambda\in{s}(F),\Im(\lambda)\in(-\pi,\pi].$$  
  \end{lem}
  \begin{proof}
  Necessarily $F=\log(e^A)$. Let $f:x\in{U}\rightarrow{e}^x\in\mathbb{C}$ where $U$ is a neighborhood of $s(F)$. Then $f$ is a holomorphic function that is one to one on $U$ and such that $f'$ is not zero on $U$. According to \cite[Theorem 2]{5}, $F$ is a polynomial in $e^F=e^A$. Therefore $F$ is a polynomial in $A$. Let $\Delta=A-F$. One has $AF=FA$ and $e^{\Delta}=e^Ae^{-F}=I_n$.  
  \end{proof}
  \begin{rem}
Note that $s(F)$ is $2i\pi$ CF, $\Delta$ is diagonalizable and $s(\Delta)\subset{2}i\pi\mathbb{Z}$	.  
  \end{rem}
 The following two results concern the equation 
 $$e^{A+B}=e^Ae^B=e^Be^A$$
  in dimension $3$.
  \begin{prop} \label{couq} 
  Let $(A,B)$ be a pair of $3\times{3}$ complex matrices such that $e^{A+B}=e^Ae^B=e^Be^A$ and $AB\not=BA$. If $\mathbb{C}^3$ is an indecomposable $<A,B>$ module, then there exist a complex number $\sigma$ and two $3\times{3}$ complex matrices $\Delta$ and $F$, that are polynomials in $A$, such that $A=\sigma{I}_3+\Delta+F$ with $e^{\Delta}=I_3,F^2=0_3$. In the same way, $B=\tau{I}_3+\Theta+G$ with $e^{\Theta}=I_3,G^2=0_3$. Moreover $FG=GF$.
    \end{prop}
  \begin{rem} It can be derived from \cite[Case (I) p. 165-166]{2}. However we give an alternative proof.
    \end{rem}
     \begin{proof}
  According to \cite{8}, $s(A),s(B)$ are not $2i\pi$ CF. Moreover the equality $$e^{A+B}e^{-A}=e^{-A}e^{A+B}=e^B$$
   implies that $s(A+B)$ is not $2i\pi$ CF. By Lemma \ref{MN}, one has $A=F+\Delta,B=G+\Theta$ where $e^F=e^A,e^G=e^B$. Thus $e^Fe^G=e^Ge^F$. According to \cite{7}, $FG=GF$. It remains to show that there exists a complex number $\sigma$ such that $(F-\sigma{I}_3)^2=0_3$. \\
 Assume that the minimal polynomial of $F$ has degree $3$, that is, in the Jordan normal form of $F$, there is exactly one Jordan block associated to each eigenvalue of $F$. Obviously $s(A)$ is $2i\pi$ CF, and one obtains a contradiction. Since $s(A)$ is not $2i\pi$ CF, $F$ has an eigenvalue with multiplicity at least $2$ and its minimal polynomial is of degree at most two. We may assume that $s(F)=\{0,0,*\}$. Up to similarity, $F$ is one of the following three forms.
\begin{eqnarray*} F&=&\begin{pmatrix}0&0&0\\0&0&0\\0&0&\lambda\end{pmatrix} \text{ where } \lambda\not={0},\\
 F&=&0_3\\
 \text{ or } F&=&\begin{pmatrix}0&1&0\\0&0&0\\0&0&0\end{pmatrix}. \end{eqnarray*}
In the last two cases, we are done. It remains to show that if $F=\begin{pmatrix}0&0&0\\0&0&0\\0&0&\lambda\end{pmatrix}$, where $\lambda\not={0}$,
then we obtain a contradiction. Note that 
 $$e^{F+G}=e^Fe^G=e^Ae^B=e^{A+B}.$$
  Therefore, if $s(F+G)\subset{(}-\pi,\pi]$, then $F+G=\log(e^{A+B})$. Clearly $F+G$ has also an eigenvalue with multiplicity at least $2$ and its minimal polynomial is of degree at most two. The matrices $F,G$ commute and, in the same way than for $F$, we can prove that $G$ is similar to one of the previous three forms. We obtain three possible values\\
$Case\; 1:$ $G=\begin{pmatrix}0&0&0\\0&0&0\\0&0&z\end{pmatrix}$. Then $\mathbb{C}^3$ is decomposable.\\ 
$Case\; 2:$  $G=\begin{pmatrix}0&1&0\\0&0&0\\0&0&0\end{pmatrix}$. One has $F+G=\log(e^{A+B})$ but its minimal polynomial is of degree $3$, that is a contradiction.\\
$Case\; 3:$ $G=\begin{pmatrix}\nu&0&0\\0&0&0\\0&0&0\end{pmatrix}$ where $\nu\not=0$. We have $F+G=\log(e^{A+B})$ and necessarily $\nu=\lambda$. Moreover $s(F+G)$ is $2i\pi$ CF and $e^{F+G}=e^{F+G+\Delta+\Theta}$. According to \cite{6}, $F+G$ and $\Delta+\Theta$ commute. We conclude that $\Delta$ and $\Theta$ are diagonal matrices and that $AB=BA$. That is a contradiction.  
    \end{proof}      
  \begin{prop}     \label{decomp}
Let $(A,B)$ be a pair of $3\times{3}$ complex matrices such that 
$$e^{A+B}=e^Ae^B=e^Be^A,$$
 $AB\not=BA$ and such that $\mathbb{C}^3$ is an indecomposable $<A,B>$ module. Then the pair $(A,B)$ has the property\\
$(*)$ The Jordan-Chevalley decompositions of $A,B,A+B$  are in the form \begin{eqnarray} \label{JC1} A&=&(\sigma{I}_3+\Delta)+F,\\ \label{JC2} B&=&(\tau{I}_3+\Theta)+G,\\ \label{JC3} A+B&=&     ((\sigma+\tau)I_3+\Delta+\Theta)+(F+G)\end{eqnarray}
 with the following equalities 
 \begin{eqnarray*} F^2&=&G^2=FG=GF=0_3,\\
  e^{\Delta}&=&e^{\Theta}=e^{\Delta+\Theta}=I_3\\
  and\; [F,\Theta]&=&[\Delta,G]. \end{eqnarray*}
  Conversely, if the pair $(A,B)$ has property $(*)$, then $e^{A+B}=e^Ae^B=e^Be^A$.   
  \end{prop}
  \begin{proof}
 We use the notations and results of Proposition \ref{couq}. Note that $\sigma{I}_3+\Delta$ is diagonalizable, $F$ is nilpotent and these matrices are polynomials in $A$. Thus (\ref{JC1}) and (\ref{JC2}) are the Jordan-Chevalley decompositions of $A,B$. Moreover 
 \begin{eqnarray*} e^A&=&e^{\sigma}(I_3+F),\\
  e^B&=&e^{\tau}(I_3+G),\\  
 \text{ and } e^{A+B}&=&e^{\sigma+\tau}(I_3+F+G+FG) \end{eqnarray*}
   with $FG=GF$. Thus $F+G+FG$ is nilpotent. According to the proof of Proposition \ref{couq}, $A+B=(\omega{I}_3+\Sigma)+O$ with $O\Sigma=\Sigma{O},e^{\Sigma}=I_3,O^2=0_3$. One has $e^{A+B}=e^{\omega}(I_3+O)$ and then $e^{\omega}=e^{\sigma+\tau},O=F+G+FG$. Finally $O^2=0_3$ implies that $FG=0_3$ and (\ref{JC3}) is the Jordan-Chevalley decomposition of $A+B$.
    Since $\Delta+\Theta$ and $F+G$ commute, one has $[F,\Theta]=[\Delta,G]$. Obviously $e^{\Delta+\Theta}=I_3$.\\
The last assertion is clear. 
   \end{proof}
  Our main result, in dimension two, is as follows
  \begin{thm} \label{dim2} 
  Let $(A,B)$ be a pair of $2\times{2}$ complex matrices. Then $(A,B)$ satisfies Condition (\ref{expbis}) if and only if $e^{A+B}=e^Ae^B=e^Be^A$ and $(A,B)$ has property L.       
  \end{thm}
  \begin{proof} $(\Rightarrow)$ There exists $t_0\in\mathbb{N}$ such that Condition (\ref{expbis}) holds for every $t\geq{t}_0$. According to Proposition \ref{seguins}, the pair $(t_0A,B)$ has property L and $(A,B)$ too.\\
  $(\Leftarrow)$ Suppose $AB\not=BA$. According to \cite{8}, $s(A)$ and $s(B)$ are not $2i\pi$ CF and, since $n=2$, $A,B$ are diagonalizable. An homothety can be added to $A$ or $B$ and we may assume $A=\begin{pmatrix}2i\pi\lambda&0\\0&0\end{pmatrix}$, $s(B)=\{2i\pi\mu,0\}$, where $\lambda,\mu\in\mathbb{Z}^*$. Again since $n=2$, $A$ and $B$ are $ST$, that is, they have a common eigenvector. Thus we may assume $B=\begin{pmatrix}2i\pi\mu&1\\0&0\end{pmatrix}$ (eventually replacing $\lambda$ with $-\lambda$ or $\mu$ with $-\mu$).  
    Note that $e^Ae^B=e^{A+B}$ if and only if $\lambda+\mu\not=0$. If $t\in\mathbb{N}$, we obtain 
   $$e^{tA}e^B=e^Be^{tA}=e^{tA+B},$$
    but eventually if $t=-\mu/\lambda$.  
  \end{proof}
  \begin{rem} The pair $A=i\pi\begin{pmatrix}1&0\\0&-1\end{pmatrix},B=\pi\begin{pmatrix}-11i&6\\16&11i\end{pmatrix}$ satisfies the condition $e^{A+B}=e^Ae^B=e^Be^A$ but has not property L. 
  \end{rem}
  We prove our main result in dimension $3$.
     \begin{thm} \label{conj} Let $(A,B)$ be a pair of $3\times{3}$ complex matrices. Then $(A,B)$ satisfies Condition (\ref{expbis}) if and only if $e^{A+B}=e^Ae^B=e^Be^A$ and $(A,B)$ has property L.    
    \end{thm}
    \begin{proof} $(\Rightarrow)$ Use the same argument than in the proof of the necessary condition of Theorem \ref{dim2}.\\
 $(\Leftarrow)$ Assume that the pair $(A,B)$ has property L, $AB\not=BA$ and 
 $$e^{A+B}=e^Ae^B=e^Be^A.$$
  $\bullet$ If $\mathbb{C}^3$ is a decomposable $<A,B>$ module, we conclude using Theorem \ref{dim2}.\\
 $\bullet$  Now $\mathbb{C}^3$ is an indecomposable $<A,B>$ module.\\
 $i)$ The pair $(A,B)$ has property $(*)$. Using notations of Proposition \ref{decomp}, we obtain for every $t\in\mathbb{N}$, 
 \begin{eqnarray*} e^{tA}&=&e^{t\sigma}(I_3+tF),\\
 e^{tA}e^B&=&e^Be^{tA}=e^{t\sigma+\tau}(I_3+tF+G),\\
 e^{tA+B}&=&e^{t\sigma+\tau}e^{t\Delta+\Theta}(I_3+tF+G). \end{eqnarray*} 
 Thus $e^{tA+B}=e^{tA}e^B=e^Be^{tA}$ if and only if $e^{t\Delta+\Theta}=I_3$.\\
 $ii)$ The pair $(\Delta+F,\Theta+G)$ has property L. We consider the associated orderings $s(\Delta+F)=s(\Delta)=(\lambda_i)_{i\leq{3}}$ and $s(\Theta+G)=s(\Theta)=(\mu_i)_{i\leq{3}}$. If $t\in\mathbb{C}$, one has \begin{eqnarray*} s(t(\Delta+F)+\Theta+G)=s((t\Delta+\Theta)+(tF+G))=(t\lambda_i+\mu_i)_{i\leq{3}}. \end{eqnarray*} 
 Since $t\Delta+\Theta$ commute with the nilpotent matrix $tF+G$, $s(t\Delta+\Theta)=(t\lambda_i+\mu_i)_{i\leq{3}}$ and the pair $(\Delta,\Theta)$ has property L.\\
 $iii)$ Since $s(\Delta)\subset{2}i\pi\mathbb{Z},s(\Theta)\subset{2}i\pi\mathbb{Z}$, if $t\in\mathbb{N}$, then $s(t\Delta+\Theta)\subset{2}i\pi\mathbb{Z}$. Thus it remains to prove that, for almost all $t\in\mathbb{N}$, $t\Delta+\Theta$ is diagonalizable. If $\Delta$ and $\Theta$ commute, we are done.\\
  We assume that $\Delta$ and $\Theta$ do not commute. Suppose that, for an infinite number of values of $t\in\mathbb{N}$, $t\Delta+\Theta$ is not diagonalizable. Then, for theses values of $t$, $(t\lambda_i+\mu_i)_{i\leq{3}}$ contains at least two equal elements. Thus, for instance, for an infinite number of values of $t$,  $t\lambda_1+\mu_1=t\lambda_2+\mu_2$. This implies that $\lambda_1=\lambda_2$ and $\mu_1=\mu_2$ and we may assume that these eigenvalues are $0$. Therefore the associated orderings are $s(\Delta)=\{0,0,\lambda\}$ where $\lambda\in{2}i\pi\mathbb{Z}^*$ and $s(\Theta)=\{0,0,\mu\}$ where $\mu\in{2}i\pi\mathbb{Z}^*$. We may assume that $\Delta=\mathrm{diag}(0,0,\lambda)$. According to \cite[Theorem 1]{4}, $$\Theta=\begin{pmatrix}W&\begin{pmatrix}u\\v\end{pmatrix}\\\begin{pmatrix}p&q\end{pmatrix}&\mu\end{pmatrix}$$
   where $W$ is a nilpotent $2\times{2}$ matrix and $u,v,p,q$ are complex numbers. We know that $\Theta$ and $\Delta+\Theta$ are diagonalizable, that is, their rank is $1$ and $\lambda+\mu\not=0$. It remains to show that, except for a finite number of values of $t\in\mathbb{N}$, $\mathrm{rank}(tA+B)=1$ and $t\lambda+\mu\not=0$.\\
 $Case\;1$. $W=\begin{pmatrix}0&1\\0&0\end{pmatrix}$. Therefore $\mathrm{rank}(\Theta)=1$ implies $p=v=0,\mu=qu$. Then $\mathrm{rank}(\Delta+\Theta)=1$ implies $\lambda=0$, a contradiction.\\
 $Case\;2$. $W=0_3$. Therefore $\mathrm{rank}(\Theta)=\mathrm{rank}(\Delta+\Theta)=1$ implies that 
 $$pu=pv=qu=qv=0.$$ 
 The previous condition implies that $\mathrm{rank}(t\Delta+\Theta)=1$, but for $t=-\mu/\lambda$.       
    \end{proof}
     
     \textbf{Acknowledgments}\\
The author thanks D. Adam for many valuable discussions.

\bibliographystyle{plain}

\end{document}